\documentclass[12pt]{article}
\usepackage{amssymb}
\usepackage{latexsym}
\usepackage[all]{xy}

\usepackage[dvips]{graphicx}
\usepackage{amssymb,color}
\usepackage[centertags]{amsmath}
\usepackage{amsfonts}
\usepackage{comment}
\usepackage{subfigure}

\usepackage{soul} % For using "\st" command (GG)

\newtheorem{thm}{Theorem}[section]

\newtheorem{lem}{Lemma}[section]

\def\BBox{\kern  -0.2cm\hbox{\vrule width 0.2cm height 0.2cm}}

\renewcommand{\Gamma}{\varGamma}
\renewcommand{\epsilon}{\varepsilon}

\renewcommand{\leq}{\leqslant}
\renewcommand{\geq}{\geqslant}

\newcommand{\E}{\mathbb{E}^3}
\begin{document}

%%%%%%%%%%%%%%%%%%%%%%%%%%%%%%%%%%%%

\title{The Regular Gr\"unbaum Polyhedron of Genus $5$}

\author{
G\'abor G\'evay \\
Bolyai Institute\\ 
University of Szeged\\
H-6720 Szeged, Hungary\\
{\it gevay@math.u-szeged.hu}\\
\and 
Egon Schulte\thanks{Supported by NSF-grant DMS--0856675}\\
Department of Mathematics\\
Northeastern University\\
Boston, MA, USA, 02115\\
{\it schulte@neu.edu}\\
\and \\[.06in]
J\"org M. Wills\\
Mathematics Institute\\
University of Siegen\\
D-57068 Siegen, Germany\\
{\it wills@mathematik.uni-siegen.de}\\}
\bigskip
\date{ \today }
\maketitle

\begin{abstract}
\noindent
We discuss a polyhedral embedding of the classical Fricke-Klein regular map of genus~$5$ in ordinary space $\mathbb{E}^3$. This polyhedron  was originally discovered by Gr\"unbaum in 1999, but was recently rediscovered by Brehm and Wills. We establish isomorphism of the {\it Gr\"unbaum polyhedron\/} with the Fricke-Klein map, and confirm its combinatorial regularity. The Gr\"unbaum polyhedron is among the few currently known geometrically vertex-transitive polyhedra of genus $g\geq 2$, and is conjectured to be the only vertex-transitive polyhedron in this genus range that is also combinatorially regular. We also contribute a new vertex-transitive polyhedron, of genus $11$, to this list, as the 7th known example. In addition we show that there are only finitely many vertex-transitive polyhedra in the entire genus range $g\geq 2$.
\medskip

\end{abstract}

\section{Introduction}
\label{intro}

Combinatorially regular polyhedra in Euclidean $3$-space $\mathbb{E}^3$ are polyhedral embeddings in $\mathbb{E}^3$ of regular maps (cell-complexes) on orientable compact closed surfaces. Regular maps on surfaces have been studied from combinatorial, topological, algebraic and geometric viewpoints for well over 100 years (see Coxeter \& Moser~\cite{cm}). Combinatorially regular polyhedra and their underlying topological maps can generally be viewed as higher-genus analogues of the Platonic polyhedra; the latter are precisely the regular maps on the $2$-sphere, of genus $g=0$. For small genus $g$, with $2\leq g\leq 6$, only eight regular maps are known to admit polyhedral embeddings with convex faces, and it is conjectured that no others occur (see \cite{swisym}). This list of eight includes famous maps of Klein, Fricke, Dyck, and Coxeter, including the classical Fricke-Klein map of genus~$5$ from 1890 (see~\cite{bok,bre,c1,cm,d2,acop,kl1,kf,mcco,swifk,swics,schw}). Some remarkable infinite series of combinatorially regular convex-faced polyhedra have also been described in the literature (see \cite{jr,msw,msw1,msw2,rz}). 

In the present paper we discuss the {\em Gr\"unbaum polyhedron\/}, a polyhedral embedding in $\mathbb{E}^3$ of the Fricke-Klein map found by Gr\"unbaum~\cite{acop} in~1999. This polyhedron has a rather interesting history of discovery. A first polyhedral realization of the Fricke-Klein map, with high symmetry but with self-intersections (thus not an embedding), was given in the 1987 paper~\cite{swikp}. Three years earlier, Gr\"unbaum and Shephard~\cite{gs} had discovered five geometrically vertex-transitive convex-faced polyhedra of genus $g\geq 2$, including an equivelar (locally regular) polyhedron with octahedral rotation symmetry and with an underlying map that shared significant combinatorial data (type, genus, number of vertices) with the Fricke-Klein map but was not isomorphic to it. Then, in 1999, Gr\"unbaum~\cite[p.\ 41, Fig.~19]{acop} described a new vertex-transitive polyhedron with octahedral rotation symmetry closely related to the previously constructed equivelar example, and announced this to be a realization of the Fricke-Klein map further to be investigated in a forthcoming article. In 2010, Brehm and Wills rediscovered the Gr\"unbaum polyhedron independently, initially planning to write a joint article with Gr\"unbaum that was supposed to contain proofs for the regularity of the polyhedron and for the isomorphism with the Fricke-Klein map. Unfortunately, neither the Gr\"unbaum article nor the Brehm-Gr\"unbaum-Wills article was ever written. However, we believe that the Gr\"unbaum polyhedron and its relationship with the Fricke-Klein map are significant enough to merit separate publication. 

In this paper we describe the geometric construction of the Gr\"unbaum polyhedron; provide integer coordinates for the vertices; supply the missing proof of isomorphism with the Fricke-Klein map, and hence of combinatorial regularity, based on the planar diagram for the map shown in \cite{swikp}; and establish that its octahedral rotation symmetry is maximum possible. We also discuss the currently known geometrically vertex-transitive polyhedra of genus $g\geq 2$, all of which occur in enantiomorphic (chiral) pairs; contribute a new vertex-transitive polyhedron, of genus $11$, to this list, as the 7th known example; and establish that there are only finitely many vertex-transitive polyhedra in the entire genus range $g\geq 2$. The Gr\"unbaum polyhedron is conjectured to be the only vertex-transitive polyhedron of genus $g\geq 2$ that is also combinatorially regular. We also present a number of attractive computer-generated pictures. 
 
\section{Basic notions}
\label{basno}

In this paper, a {\em polyhedron\/} $P$ is a compact closed surface in Euclidean $3$-space $\E$ made up of finitely many {\em convex\/} polygons, the {\em faces\/} of $P$, such that any two distinct polygons intersect, if at all, in a common vertex or a common edge (see Brehm \& Wills~\cite{brwil}, Brehm \& Schulte~\cite{brs}). Thus $P$ is free of self-intersections. The vertices and edges of the faces of $P$ are called the {\em vertices\/} and {\em edges\/} of $P$, respectively. We require that no  two {\em adjacent\/} faces (with a common edge) lie in the same plane. For notational convenience we usually identify $P$ with the map on the underlying orientable surface, or with the abstract polyhedron consisting of the vertices, edges, and faces, partially ordered by inclusion (see Coxeter \& Moser~\cite{cm}, McMullen \& Schulte~\cite{arp}). 

A polyhedron $P$ is said to be {\em combinatorially regular\/} if its combinatorial automorphism group $\Gamma(P)$ is transitive on the flags (incident triples consisting of a vertex, an edge, and a face) of $P$. A combinatorially regular polyhedron is a polyhedral embedding in $\E$ of a finite (regular) map on an orientable surface (see \cite{cm}). A map (or polyhedron, resp.) is called {\em equivelar\/}, of ({\em Schl\"afli\/}) {\em type\/} $\{p,q\}$ (with finite $p,q\geq 3$), if the faces are $p$-gons such that $q$ meet at each vertex.  Regular maps (or combinatorially regular polyhedra, resp.) are equivelar. For an equivelar map $P$ of type $\{p,q\}$ with $f_0$ vertices, $f_1$ edges, and $f_2$ faces on an orientable surface of genus $g$, the number of flags $f$ (that is, the order of $\Gamma(P)$ when $P$ is regular) and the genus are linked via
\begin{equation}
\label{chimap}
2-2g = f_{0} - f_{1} + f_{2} = \frac{f}{2}\,(\frac{1}{p} + \frac{1}{q} - \frac{1}{2}).
\end{equation}

It is well-known that the Platonic solids are the only regular maps on the $2$-sphere (with $g=0$), and that there are infinitely many regular maps on the $2$-torus (with $g=1$), each of type $\{3,6\}$, $\{6,3\}$ or $\{4,4\}$. Each surface of genus $g\geq 2$ supports at most a finite number of regular maps, but there is an infinite number of surfaces that do not admit a regular map with a simple edge graph at all (see Conder, Siran \& Tucker~\cite{cst}, Breda d'Azevedo, Nedela \&~Siran~\cite{bns}). We are particularly interested in surfaces of small genus. The Klein map $\{3,7\}_8$ and Dyck map $\{3,8\}_6$ are the most prominent examples for genus $3$. The Fricke-Klein map of type $\{3,8\}$ and genus $5$ is a $2$-fold covering of Dyck's map; the former is the map associated with the Gr\"unbaum polyhedron and is shown in Figure~\ref{FKkarte}. See Conder~\cite{con} for a complete census of regular maps on orientable surfaces of genus up to~101. 

\begin{figure}[htbp]
\centering
    \includegraphics[width=.5\textwidth]{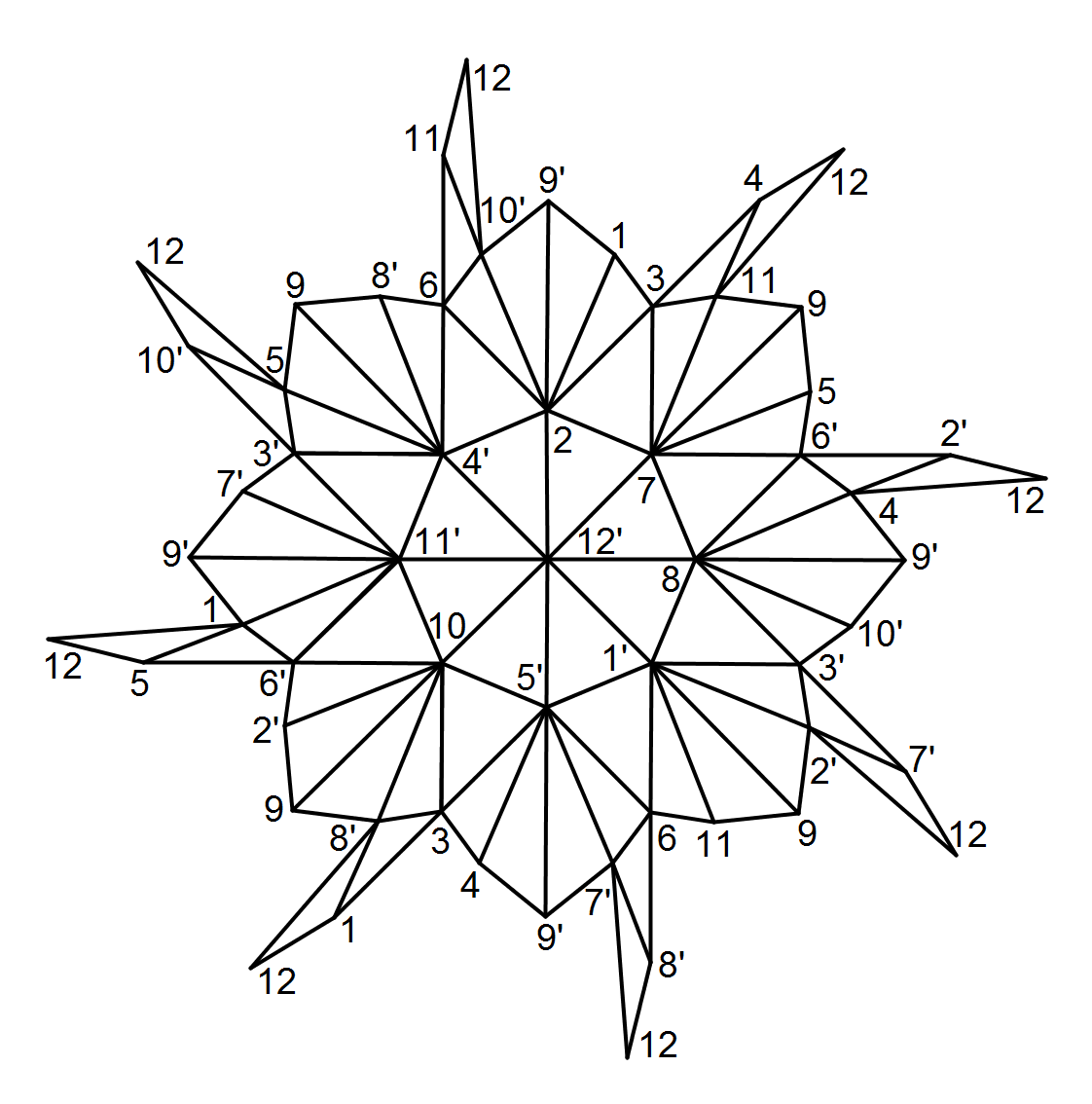}
    \medskip
    \caption{The Fricke-Klein map of type $\{3,8\}$ and genus $5$.} 
    \label{FKkarte}               
\end{figure}

For a regular map $P$ of type $\{p,q\}$, the group $\Gamma(P)$ is generated by ``combinatorial reflections" $\rho_0,\rho_1,\rho_2$ in the sides of a fundamental triangle that satisfy (at least) the {\em Coxeter relations\/} 
\begin{equation}
\label{coxrel}
\rho_{0}^2 = \rho_{1}^{2} = \rho_{2}^{2} = (\rho_0\rho_1)^{p} = (\rho_1\rho_2)^{q} = (\rho_0\rho_2)^{2} = 1. 
\end{equation}
For a complete presentation of $\Gamma(P)$ extra relations are needed precisely when $g>0$. For example, the addition of the {\em Petrie relation\/} 
\begin{equation}
\label{exrel1}
(\rho_0\rho_1\rho_2)^{r} = 1, 
\end{equation} 
with $r=8$ or $6$, to the relations in (\ref{coxrel}) suffices for a complete presentation of the automorphism groups of the Klein map and the Dyck map, respectively. The Coxeter element $\rho_0\rho_1\rho_2$ of $\Gamma(P)$ occurring in (\ref{exrel1}) shifts a certain Petrie polygon of $P$ one step along itself, and hence has period $r$ if the Petrie polygon has length $r$. Recall here that a {\em Petrie polygon\/} of a regular map is a zigzag along the edges such that every two, but no three, successive edges lie in a common face. The {\em Petrie dual\/} of $P$ is a new regular map (on a generally different surface), with the same edge graph as $P$ but the faces given by the Petrie polygons of $P$.  

The relations in (\ref{coxrel}) are implying that every regular map of type $\{p,q\}$ is a quotient of the corresponding regular tessellation $\{p,q\}$ of the $2$-sphere (if $g=0$), the Euclidean plane (if $g=1$), or the hyperbolic plane (if $g\geq 2$). If a map is the quotient of the regular tessellation $\{p,q\}$ obtained by identifying any two vertices that are separated by $r$ steps along a Petrie polygon, for a specified value of $r$, then the map is denoted $\{p,q\}_r$. The Klein map and the Dyck map are examples of this kind.

The geometric symmetry group $G(P)$ of a polyhedron $P$ can be viewed as a subgroup of $\Gamma(P)$. Generally $G(P)$ is small compared with $\Gamma(P)$. Naturally, when searching for polyhedral realizations of a given regular map we are most interested in those that exhibit maximum possible geometric symmetry.

\section{The Gr\"unbaum polyhedron}
\label{grupo}

The {\em Gr\"unbaum polyhedron\/} is the polyhedral embedding of the Fricke-Klein map of type $\{3,8\}$ and genus $5$ discovered by Gr\"unbaum~\cite{acop}; as mentioned earlier, the polyhedron was rediscovered by Brehm and Wills (in unpublished work).

\medskip
\noindent
{\bf Construction}
\smallskip

The Gr\"unbaum polyhedron and the Fricke-Klein map each have $24$ vertices, $96$ edges, and $64$ faces, as well as an automorphism group of order $384$ (see \cite{kf}). The Petrie polygons all have length $12$; thus the Fricke-Klein map is a finite quotient of the infinite regular map $\{3,8\}_{12}$ (see \cite[p.\,399]{arp}). On the other hand, the Fricke-Klein map doubly covers Dyck's map $\{3,8\}_6$ of genus~$3$ via a covering mapping determined by a central involution in the automorphism group of the Fricke-Klein map; this covering relationship is particularly nicely revealed on the polyhedral realizations with self-intersections for these two maps described in \cite{swidy,swikp}.

\begin{figure}[h!] 
\centering
    \mbox{\hskip -3pt
    \subfigure[The full polyhedron.]{
    \includegraphics[width=.4\textwidth]{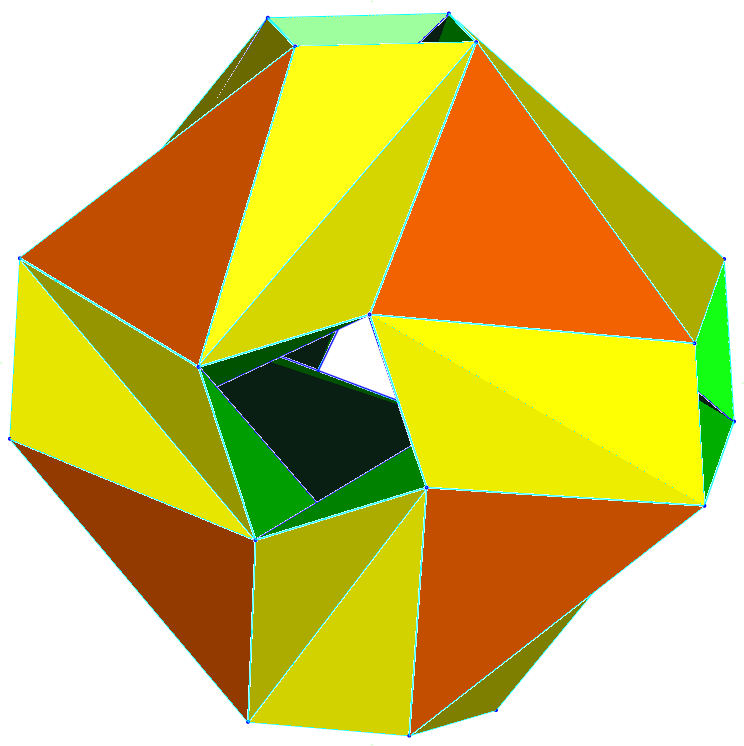}}} 
    \vskip.2in
    \mbox{\hskip -3pt
    \subfigure[The outer shell.]{
    \includegraphics[width=.4\textwidth]{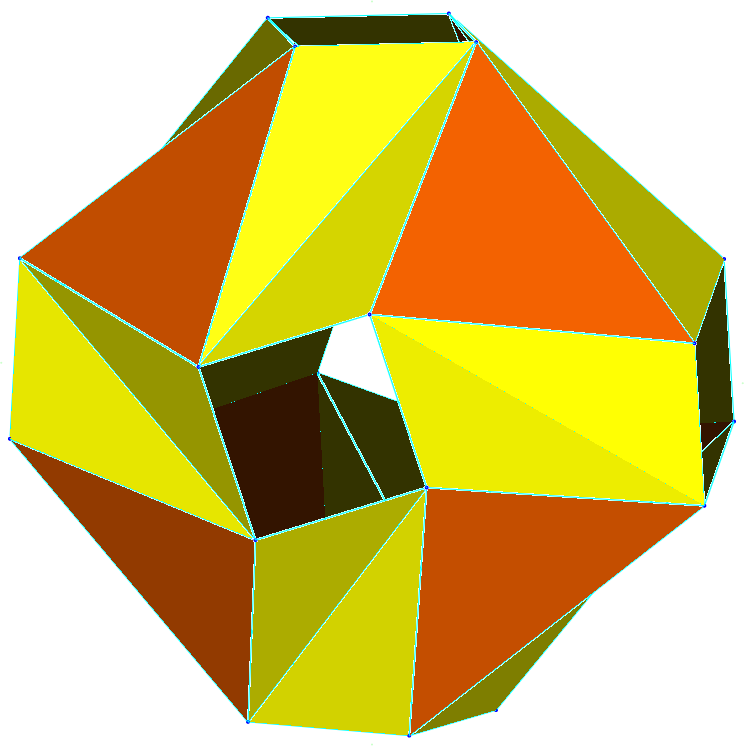}} 
    \subfigure[The inner shell.]{
    \includegraphics[width=.4\textwidth]{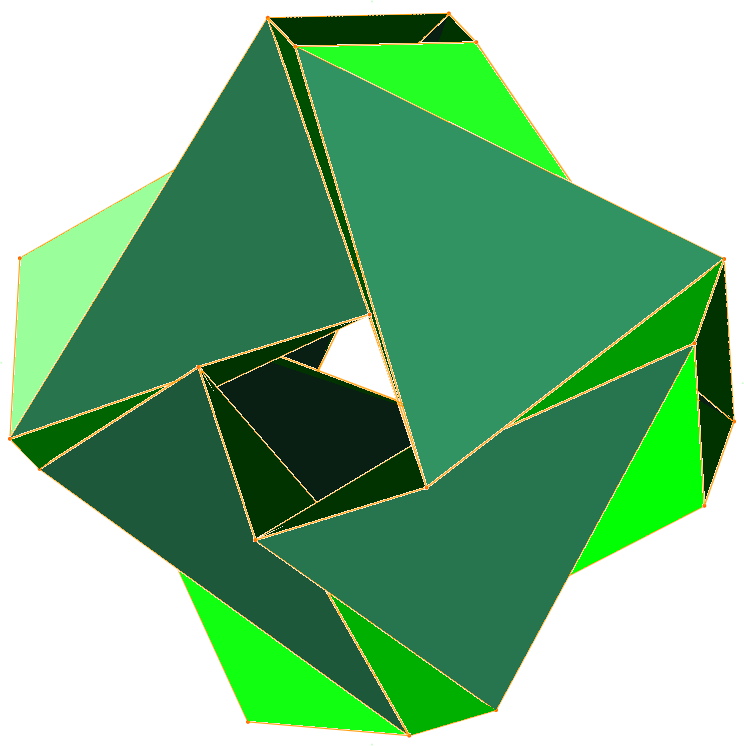}}}
    \caption{Gr\"unbaum's regular polyhedron of genus 5.}              
    \label{GrunbaumPolyhedron}
\end{figure}

The Gr\"unbaum polyhedron $P$ (say) shown in Figure~\ref{GrunbaumPolyhedron} has octahedral rotation symmetry and can best be described in terms of the geometry of the Archimedean snub cube. The $24$ vertices of $P$ are those of the snub cube, but only $8$ triangle faces of the snub cube are retained as faces of $P$. The polyhedron consists of an outer shell $P^{o}$ and an inner shell $P^i$ connected precisely at the square ``holes" formed by the square face boundaries of the snub cube (but nowhere else). The entire polyhedron $P$ can be pieced together from the orbits of four suitably chosen triangles under the octahedral rotation group, a pair of adjacent triangles taken from each shell. The triangles in these orbits then are the faces of $P$. 

Brehm and Wills showed that the underlying snub cube can be chosen in such a way that the vertices of $P$ have small integer coordinates. The smallest possible integer coordinates are obtained when the four particular triangles are chosen as follows. (Computer generated images of the polyhedral surfaces with smaller integer coordinates show that these surfaces have self-intersections.) The two adjacent triangles for the outer shell have vertex sets 
\[ \{(1,2,6),(2,6,1),(6,1,2)\},\;\; \{(1,2,6),(2,6,1),(-2,1,6)\}.\]
Their images under the standard octahedral rotation group (generated by 4-fold rotations about the coordinate axes and 3-fold
rotations about the main space diagonals) comprise 8~regular triangles and 24 non-regular triangles forming the outer
shell $P^o$. The inner shell $P^i$ similarly consists of 8~regular triangles and 24 non-regular triangles obtained under the standard octahedral rotation group from the two adjacent triangles for $P^i$ with vertex sets 
\[ \{(2,-1,6),(-1,6,2),(6,2,-1)\},\;\; \{(2,-1,6),(-1,6,2),(-2,6,-1)\}.\]
It is immediately clear by construction that the symmetry group $G(P)$ is vertex-transitive.

\medskip
\noindent
{\bf Combinatorial regularity}
\smallskip

In order to prove isomorphism of the Gr\"unbaum polyhedron $P$ with the Fricke-Klein map, and thus establish combinatorial regularity for $P$, we employ the planar diagram for the Fricke-Klein map shown in Figure~\ref{FKkarte} (and taken from \cite[Fig. 12]{swikp}). To this end, we begin by producing in Figure~\ref{diagrams} two diagrams (reminiscent of Schlegel diagrams) which accurately represent the combinatorics of the face decompositions of the inner shell and outer shell, respectively (see Figure~\ref{GrunbaumPolyhedron}). Each diagram consists of a large square, the {\em outer frame\/}, and five smaller squares inside; its $24$ vertices are labeled $1,\ldots,12$ and $1',\ldots,12'$ as indicated (using the same labeling for both diagrams). The squares represent the six square holes of $P$ determined by the square faces of the snub cube (the squares themselves are not faces of $P$). On each diagram, the space between the inner squares and the outer frame is (topologically) triangulated in exactly the same way in which $P^o$ or $P^i$, respectively, is triangulated by the triangular faces of $P$ lying in $P^o$ or $P^i$. In our mind, we then can obtain an accurate picture of the combinatorics of the entire polyhedron $P$ by amalgamating the two diagrams along corresponding square holes. With this in place, it is now straightforward to verify that the vertex-labeled topological triangles on the two diagrams in Figure~\ref{diagrams} match precisely the faces on the planar diagram for the Fricke-Klein map shown in Figure~\ref{FKkarte}. Thus the Gr\"unbaum polyhedron is a polyhedral embedding of the Fricke-Klein map, and in particular is combinatorially regular since the map is regular.

\begin{figure}[htbp] 
\centering
    \mbox{
    \hskip -3pt
    \subfigure[The diagram for the outer shell.]{
    \includegraphics[width=.45\textwidth]{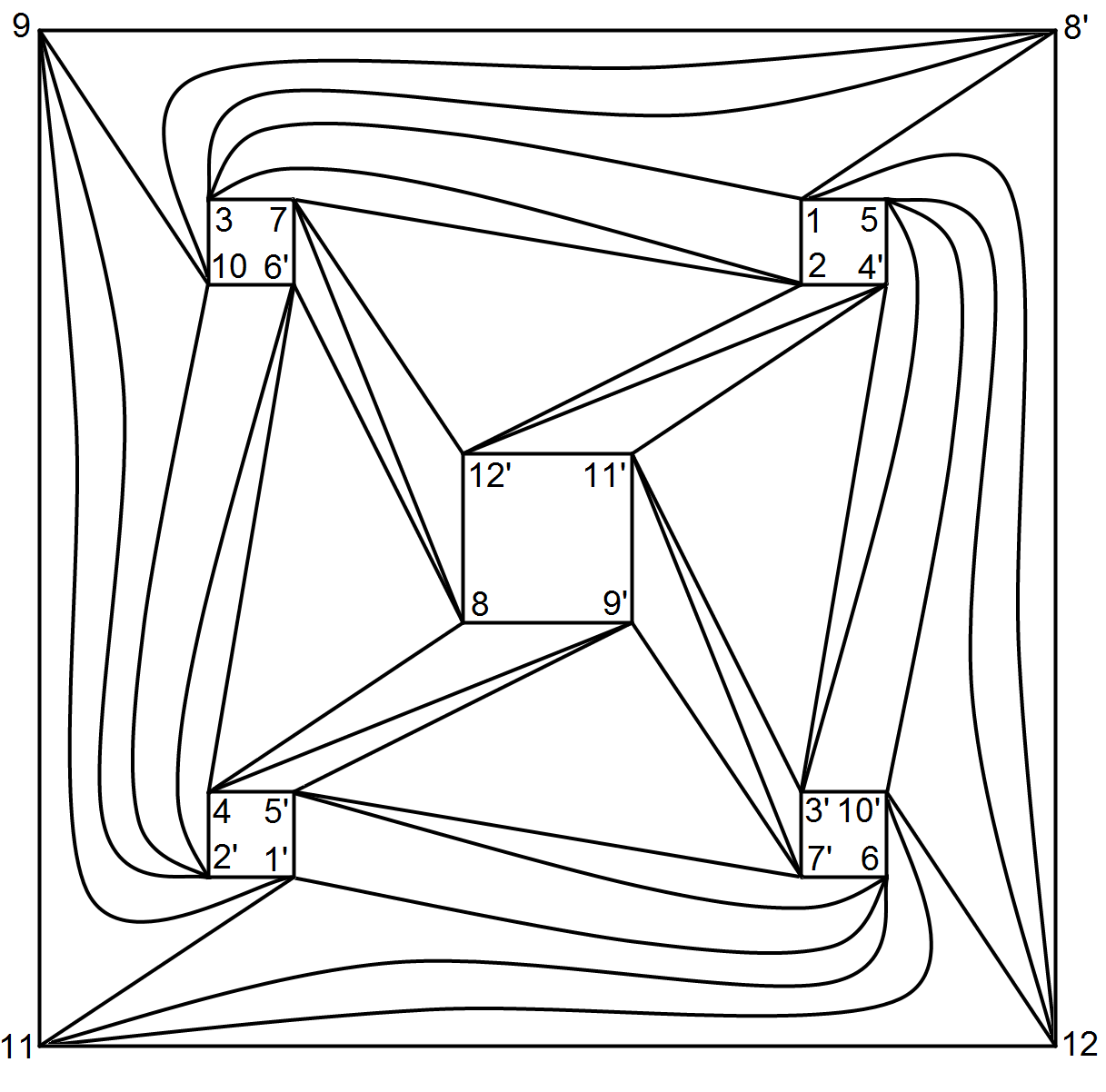}}} 
    \mbox{
    %\hskip -3pt
    \subfigure[The diagram for the inner shell.]{
    \includegraphics[width=.45\textwidth]{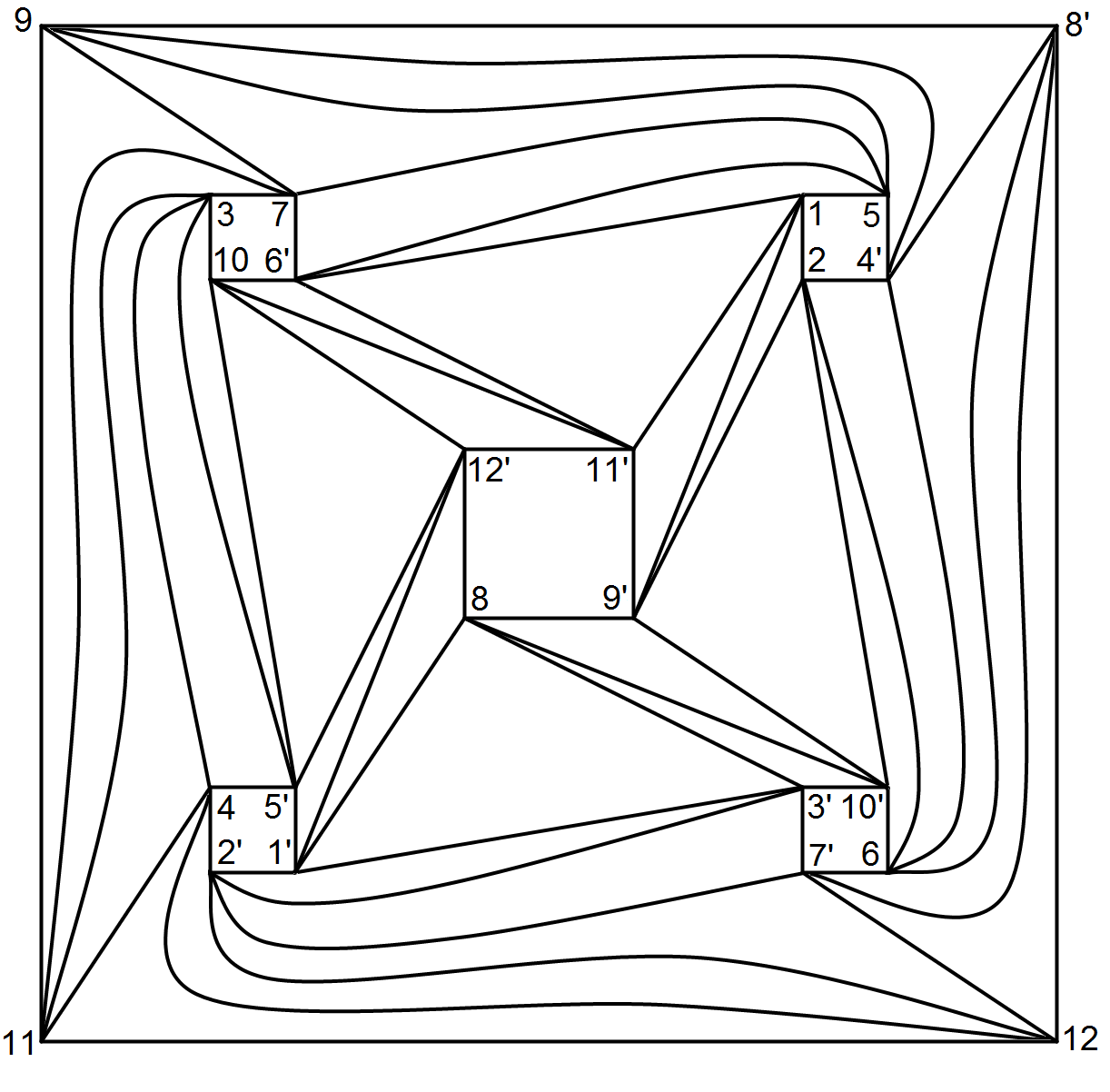}}}
    \caption{The face decomposition of the two shells of the Gr\"unbaum polyhedron.}                
    \label{diagrams}
    \end{figure}

Note that the particular vertex labeling in our diagrams is borrowed from the vertex-labeling for the polyhedral realization with
self-intersections for the Fricke-Klein map depicted in \cite[Fig.\,10]{swikp}; this realization was built from a pair of homothetic icosahedra. In particular, the permutation that pairs up the vertices $i$ and $i'$ for each $i=1,\ldots,12$, defines a central involution $\gamma$ in the automorphism group $\Gamma(P)$ of $P$; the vertices $i$ and $i'$ then are {\em antipodal\/} relative to $\gamma$. Note in particular that when we identify pairs of antipodal vertices (relative to $\gamma$) in the Gr\"unbaum polyhedron, we arrive at the polyhedral realization with self-intersections for Dyck's map described in \cite{swidy}. This observation highlights the fact that the Fricke-Klein map is a double cover of Dyck's map. Note that the Gr\"unbaum polyhedron is not (geometrically) centrally symmetric in $\mathbb{E}^3$, so antipodality relative to $\gamma$ should not be confused with antipodality relative to a point symmetry in $\E$.

\medskip
\noindent
{\bf Interesting geometric properties}
\smallskip

We first show that octahedral rotation symmetry, as exhibited by the Gr\"unbaum polyhedron, is the maximum possible symmetry for any polyhedral embedding of the Fricke-Klein map in $\mathbb{E}^3$. Here we prefer to give a direct proof although the result is implied by Theorem~\ref{verpolfin} in the next section.

Now it is clear that, because of the group order, only the full octahedral group must be excluded as a possible symmetry group. To this end, suppose a polyhedral embedding $P'$ (say) of the Fricke-Klein map has a full octahedral group as its symmetry group. Then the vertex set of $P'$ must be the disjoint union of point orbits under the full octahedral group, and these orbits must necessarily have sizes $1$, $6$, $8$, $12$, $24$, or $48$. The only two options are one orbit of size $24$, and two orbits of size $12$. In the former case there are exactly three mutually non-equivalent ways to choose a representative point for the single orbit; accordingly, the vertices of $P'$ must coincide with those of a truncated cube, truncated octahedron, or rhombi-cuboctahedron.  In the latter case, the vertices of $P'$ lie on two orbits each given by the vertex-set of a cuboctahedron in a pair of concentric cuboctahedra. On the other hand, it is not difficult to see that a polyhedral embedding of the Fricke-Klein map (with full octahedral symmetry) cannot have any of these vertex arrangements, proving that $P'$ cannot exist. Thus octahedral rotation symmetry is maximum possible. 

The Gr\"unbaum polyhedron is a {\em Leonardo polyhedron\/} in the sense of \cite{gwi}, meaning that the symmetry group of the polyhedron is either the full symmetry group of a Platonic solid or its rotation subgroup. It seems to be challenging to find Leonardo polyhedra with large automorphism groups, particularly Leonardo polyhedra which are also combinatorially regular. The Gr\"unbaum polyhedron is one of the few known examples of this kind.

The Fricke-Klein map, being regular and combinatorially centrally-symmetric (with the central symmetry defined by $\gamma$), also admits a polyhedral realization in the $2$-skeleton of the $12$-dimensional crosspolytope in Euclidean $12$-space $\mathbb{E}^{12}$ in which all combinatorial symmetries of the map are realized by geometric symmetries of the polyhedral realization in $\mathbb{E}^{12}$. (The $12$-dimensional crosspolytope is the regular convex polytope in $\mathbb{E}^{12}$ whose $24$ vertices are given by the $12$ canonical basis vectors and their negatives.) This polyhedral realization can either be constructed directly from the map data, or seen as an interesting special case of a more general theorem about realizations of combinatorially centrally symmetric abstract regular polytopes (see \cite[pp. 135-136]{arp}). Note that all the faces in this realization are equilateral triangles.

%\bigskip
\noindent
{\bf A close relative}
\smallskip

The Gr\"unbaum polyhedron $P$ is closely related to another geometrically vertex-transitive polyhedron of type $\{3,8\}$ and genus $5$ shown in Figure~\ref{RelativeGrunbaum}. This polyhedron $Q$ (say) is not combinatorially isomorphic to $P$. It is among the five geometrically vertex-transitive polyhedra of genus $g\geq 2$ described in Gr\"unbaum \& Shephard~\cite[Fig.\,4]{gs}. These five polyhedra are Leonardo polyhedra of genus $g=3$, $5$, $7$, $11$ or $19$, whose genus is one less than the number of faces of the tetrahedron, cube, octahedron, dodecahedron or icosahedron, respectively. 

\begin{figure}[htbp] 
  \centering
    \mbox{
    %\hskip -3pt
    \subfigure[The full polyhedron.]{
    \includegraphics[width=.4\textwidth]{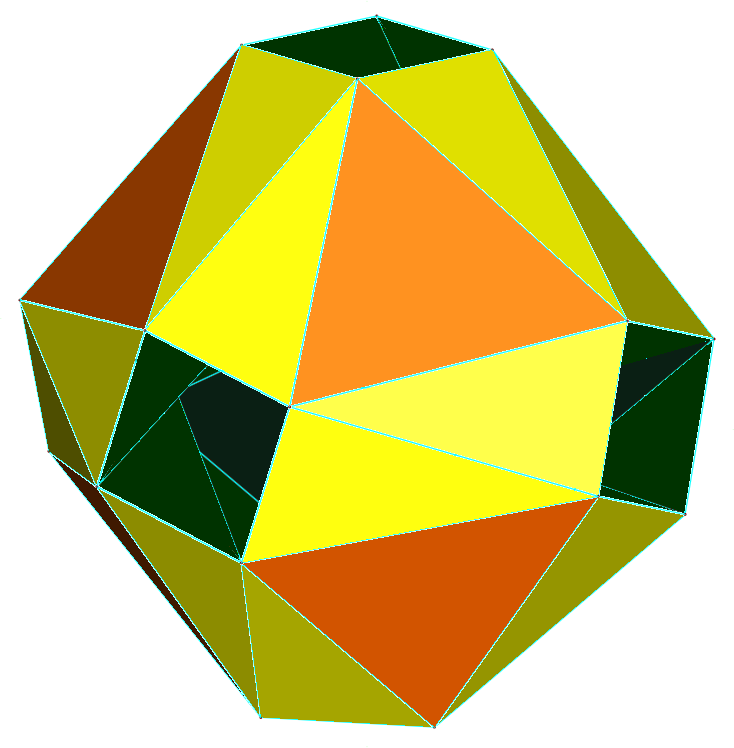}}}
    \mbox{
    %\hskip -3pt
    \subfigure[The inner shell.]{
    \includegraphics[width=.4\textwidth]{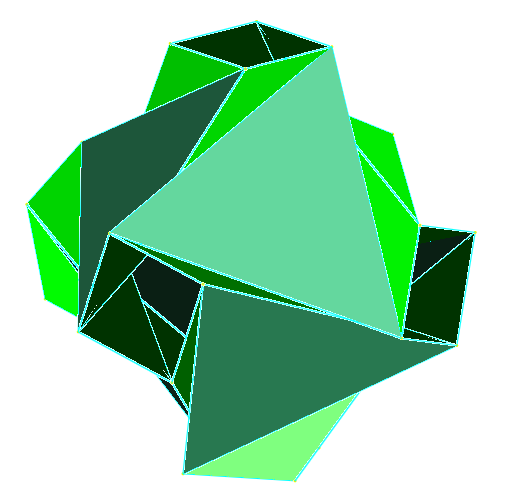}}}
    \caption{A close relative of the Gr\"unbaum polyhedron.}                
    \label{RelativeGrunbaum}
    \end{figure}
The two polyhedra $P$ and $Q$ can be constructed from each other by replacing, in the outer shells, certain $12$ pairs of adjacent triangle faces of one polyhedron by certain $12$ pairs of adjacent triangle faces of the other polyhedron, such that matching pairs share the same vertices, and form the boundary of a tetrahedron if fit together. Thus, topologically speaking, when a pair of adjacent triangle faces is viewed as a square cut in half by a diagonal on the underlying surface, then a single pair replacement is simply a switch of diagonals in this square (this switch is often called a Pachner move). In particular, $P$ and $Q$ both share the same octahedral rotation group as symmetry group, and the $12$ pairs of adjacent triangle faces in each form a single orbit under this group. It is a remarkable fact that the switch of $12$ triangle pairs does not lead to self-intersections for the newly-produced figure. Clearly, for an arbitrary polyhedron with triangle faces, switching of one or more diagonals generally  results in self-intersections. 

While the switch of face pairs on $P$ produces an aesthetically pleasing outcome, the new polyhedron $Q$ is no longer combinatorially regular. In fact, $Q$ has Petrie polygons of different length and hence cannot be combinatorially regular.

Among the other four geometrically vertex-transitive polyhedra described in~\cite{gs}, only the polyhedron of genus $11$ seems to permit similar switches of face pairs that give a new vertex-transitive polyhedron not isomorphic to the original one. Here the genus, symmetry group, and Schl\"afli type are preserved under these switches, and self-intersections can be avoided. Figure~\ref{g11} shows the original polyhedron and the new polyhedron, both of genus $11$, of type $\{3,8\}$, and with icosahedral rotation symmetry. Each occurs in two enantiomorphic forms, a right-handed and a left-handed version. Note that the polyhedra in Figure~\ref{g11} form a pair of combinatorially distinct vertex-transitive Leonardo polyhedra with the same genus, Schl\"afli type, and symmetry group. In addition to the Gr\"unbaum polyhedron and its relative, this is the only other pair of this kind known when $g\geq 2$. 

\begin{figure}[htbp] 
  \centering
    \mbox{
    %\hskip -3pt
    \subfigure[The polyhedron of \cite{gs}.]{
    \includegraphics[width=.4\textwidth]{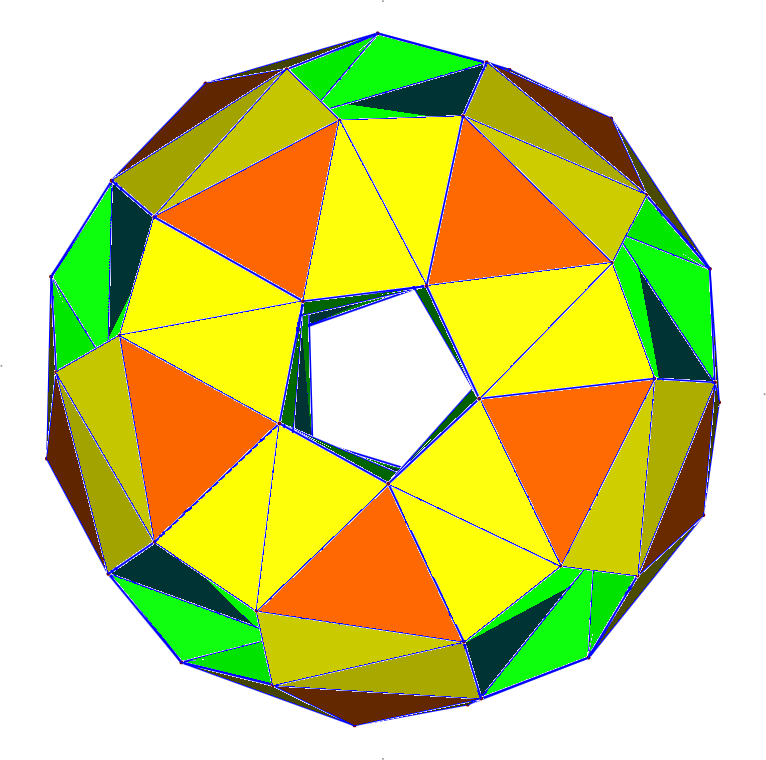}}}
    \mbox{
    %\hskip -3pt
    \subfigure[The new polyhedron.]{
    \includegraphics[width=.4\textwidth]{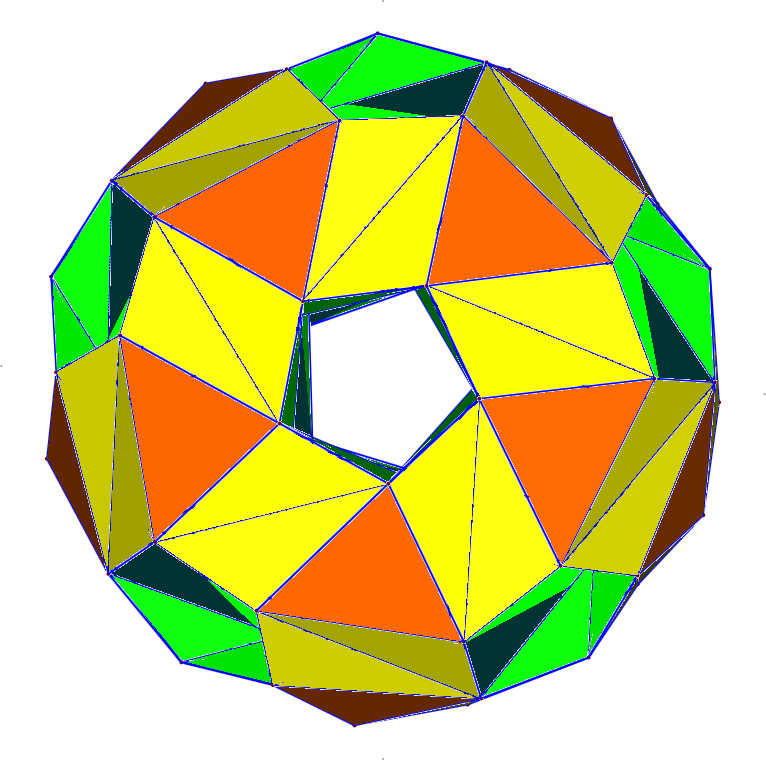}}}
    \caption{Combinatorially distinct vertex-transitive polyhedra of genus $11$ and type $\{3,8\}$.}                
    \label{g11}
    \end{figure}

\section{Vertex-transitive polyhedra}
\label{vertran}

Polyhedra with vertex-transitive symmetry groups are severely restricted in their possible geometric shapes. Gr\"unbaum and  Shephard~\cite{gs} (and independently Brehm, in unpublished work) established that there are infinitely many (isomorphism classes of) vertex-transitive polyhedra of genus $1$. However, for the genus range $g\geq 2$, only seven vertex-transitive polyhedra seem to be known. The known examples comprise the five vertex-transitive polyhedra of genus $g=3$, $5$, $7$, $11$ and $19$ discovered in \cite{gs}, as well as the Gr\"unbaum polyhedron and the polyhedron of genus $11$ shown in Figure~\ref{g11}(b). The Gr\"unbaum polyhedron is the only combinatorially regular polyhedron among these seven.

In this section we show that there are only finitely many vertex-transitive polyhedra in the genus range $g\geq 2$ (see Theorem~\ref{verpolfin} below). Note that this result concerns a whole range of genera, not any specific values of $g$. In fact, it is known that for each $g$ with $g\geq 2$ there are only finitely many vertex-transitive (polyhedral) maps on an orientable surface of genus $g$, while there are infinitely many such maps of genus $g=0$ or $1$, respectively (see \cite{kane,thom}). Thus, for each $g\geq 2$, there clearly can only be finitely many (isomorphism types of) vertex-transitive polyhedra of genus $g$; however, this does not directly translate into a corresponding statement for the entire genus range.

Clearly, the number of vertices of a geometrically vertex-transitive polyhedron is bounded by the order of its symmetry group. Hence, if we can bound the group orders for the symmetry groups from a particular class of groups, then, in effect, we have bounded the number of vertices of vertex-transitive polyhedra, and thus the total number of (isomorphism classes of) vertex-transitive polyhedra, with a symmetry group from this particular class of groups.

Now let $R$ be a vertex-transitive polyhedron with symmetry group $G(R)$. Then $G(R)$, being a finite group of isometries of $\E$, must necessarily leave a point in $\E$ invariant, and we may take this point to be the origin, $o$, of $\E$. Thus $G(R)$ is a finite (irreducible or reducible) subgroup of ${\rm O}(3)$, the orthogonal group of $\E$. This immediately limits the number of possible groups that can occur. 
\medskip

\noindent
{\bf Irreducible symmetry groups}
\medskip

An inspection of the list of the finite subgroups of ${\rm O}(3)$ shows that the symmetry groups of the Platonic solids and their rotation subgroups are the only finite irreducible subgroups of $O(3)$, except for one group; the only additional finite irreducible subgroup of $O(3)$ is the {\em pyritohedral group\/}, $U$ (say), which is isomorphic to $A_{4}\times C_2$ and obtained from the tetrahedral rotation group by adjoining the central inversion in $o$ (see \cite[Ch.\,2]{grove}). This immediately provides an upper bound for the orders of irreducible symmetry groups $G(R)$, and in particular settles the case of irreducible groups. 

Before moving on to reducible groups we note that a vertex-transitive polyhedron of positive genus with an irreducible symmetry group cannot have a plane of symmetry; that is, the full Platonic symmetry groups, as well as the exceptional group $U$, do not occur as symmetry groups of vertex-transitive polyhedra of positive genus. This is based on the following simple lemma, which later also enables us to reject the possibility of reflections in the case of reducible symmetry groups.

\begin{lem}
\label{edgecross}
Each plane of symmetry of a vertex-transitive polyhedron must be the perpendicular bisector for each edge of the polyhedron that crosses it.
\end{lem}

\begin{proof}
In fact, if an edge crosses a plane of symmetry, then its image under the plane reflection must cross the plane at the exact same point. This leads to self-intersections, unless the edge is invariant under the plane reflection.
\end{proof}

Now consider the implications of Lemma~\ref{edgecross} for vertex-transitive polyhedra of positive genus with an irreducible symmetry group. Now suppose a polyhedron $R$ of this kind has a plane of symmetry. Then either $G(R)=G(S)$ for some Platonic solid $S$, or $G(R)=U$. 

First suppose that $G(R)=G(S)$, where $S$ is a Platonic solid (with center $o$). Then each plane of symmetry of $R$ is a plane of symmetry of $S$, and vice versa. The family of all planes of symmetry of $S$ naturally subdivides $S$ into tetrahedra, each a fundamental tetrahedron for $G(S)$ in $S$. The simplicial $3$-complex $\mathcal{C}$ of all these tetrahedra is the barycentric subdivision of $S$; note that $\mathcal{C}$ can be viewed as a cone, with apex $o$, over the Coxeter complex on the boundary ${\rm bd}(S)$ of $S$ associated with the (Coxeter reflection) group $G(S)$. In particular, if $T$ is any fundamental tetrahedron, then $G(S)$ is generated by the three reflections in the walls of $T$ containing $o$. Ignoring $o$ as a trivial singleton orbit, there are (essentially) seven different kinds of point orbits in $S$ under $G(S)$; in particular, up to scaling, the vertex-set of the polyhedron $R$ must be of one of these kinds. Namely, if $T':=T\cap {\rm bd}(S)$ denotes the fundamental triangle for $G(S)$ on ${\rm bd}(S)$ determined by $T$, then, up to scaling, a representative point in $T$ of any such $G(S)$-orbit must necessarily be a vertex, a relative interior point of an edge, or a relative interior point, of $T'$.

Now, up to scaling, any vertex of $R$ must be equivalent under $G(S)=G(R)$ to a vertex $u$ of $R$ in $T'$ and hence must necessarily belong to an orbit of the kind described above. Then, by Lemma~\ref{edgecross}, the neighboring vertices $v$ of $u$ in $R$ are trapped inside the tetrahedra of $\mathcal{C}$ that do not contain $u$ but are adjacent to tetrahedra containing $u$
along a wall perpendicularly bisected by the corresponding connecting edge $\{u,v\}$ of $R$. However, this is precisely the way in which the vertices of the convex hull of the orbit of $u$ under $G(S)$ are related; note that this convex hull is just the convex hull of $R$ itself. Hence, bearing in mind that the faces of $R$ must be convex, we see that $R$ must coincide with the boundary of its own convex hull. Thus $R$ has genus $0$, contrary to our assumption.

The analysis in case $G(R)=U$ is similar. Here we may take $U$ to be generated by the reflections in the three coordinate planes 
and the $3$-fold rotations about the main space diagonals through $o$; abstractly, $U = C_{2}^3 \rtimes C_3$, a semidirect product of $C_2^3$ and $C_3$. Now if $S$ denotes the regular octahedron with vertices $(\pm 1,0,0)$, $(0,\pm 1,0)$ and $(0,0,\pm 1)$, then $U$ is a subgroup of $G(S)$ of index $2$ with fundamental tetrahedron given by $T\cup T_0$, where as above $T$ is the fundamental tetrahedron in $\mathcal{C}$ for $G(S)$, and $T_0$ is the adjacent tetrahedron in $\mathcal{C}$ meeting $T$ in the face of $T$ opposite a vertex of $S$. Then, passing to the fundamental triangle $(T\cup T_0)\cap {\rm bd}(S)$ for $U$ in ${\rm bd}(S)$, we can proceed similarly as before and trap the neighbors of a vertex in $R$ inside certain tetrahedra (copies of $T\cup T_0$ under $U$). The details are slightly more involved in this case, since $U$ contains only three plane reflections and so Lemma~\ref{edgecross} is less forceful here. Some placements of representative vertices produce larger than expected symmetry, namely full octahedral symmetry rather than confined to $U$. In any case, it turns out that $R$ must  coincide with the boundary of its own convex hull, so again $g=0$, a contradiction.

Thus the presence of a plane of symmetry in a vertex-transitive polyhedron forces it to be spherical and, in particular, form the boundary of an Archimedean solid.
\medskip

\noindent
{\bf Reducible symmetry groups}
\medskip

We now turn to reducible groups. Unlike in the irreducible case, the orders of the finite reducible subgroups of ${\rm O}(3)$ are not bounded. Our goal is to eliminate the various (large enough) reducible subgroups on geometric grounds as possible symmetry groups for vertex-transitive polyhedra of genus $g\geq 2$ (see Theorem~\ref{genus1} below). We begin by reviewing the enumeration of the finite reducible subgroups of $O(3)$. 

In the notation of \cite[Ch.\,2]{grove}, the seven different kinds of finite reducible subgroups of $O(3)$ can be described as follows. For a finite subgroup $H$ of $O^{+}(3)$, the subgroup of proper isometries of $O(3)$, we set  
\[ H^{*}:=H \cup (-i)H \;\;(= H \times \langle -i\rangle),\] 
where $i$ is the identity mapping and $-i$ is the central inversion in $o$. Moreover, if $K$ is another finite subgroup of $O^{+}(3)$ containing $H$ as a subgroup of index $2$, we define 
\[ K]H := H \cup (-i)(K\setminus H),\] 
where as usual $K\setminus H$ is the set-theoretic difference of $K$ and $H$. Then $-i$ belongs to each group of type $H^*$, but not to any group of type $K]H$. Further, let $C_n$ denote the cyclic group of order $n$ generated by a rotation by $2\pi/n$ about the $z$-axis in $\E$, and let $H_n$ denote the subgroup of $O^{+}(3)$ consisting of the rotations in $C_n$ and the half-turns about the $n$ lines of symmetry of a convex regular $n$-gon in the $xy$-plane. Then the seven kinds of finite reducible subgroups of $O(3)$ are given by 
\begin{equation}
\label{redgr}
C_n,\;C_n^*,\;C_{2n}]C_{n} \;(n\geq 1);\;\, H_n,\; H_n^*,\; H_{n}]C_{n},\; H_{2n}]H_{n}\;(n\geq 2). 
\end{equation}
The group orders, respectively, are $n$, $2n$, $2n$, as well as $2n$, $4n$, $2n$, $4n$. Each of these subgroups contains $C_n$.
 
Now suppose the symmetry group $G(R)$ of the vertex-transitive polyhedron $R$ is reducible. Then, up to conjugacy in $O(3)$, our group $G(R)$ coincides with one of the reducible  groups in (\ref{redgr}) for some $n$. Now since reducible groups of orders less than or equal to $16$ can contribute at most finitely many (isomorphisms classes of) vertex-transitive polyhedra, we can ignore small values of $n$ and assume from now that $n\geq 5$.

Next observe that the vertices of $R$ are inscribed in a sphere, since the vertex-set forms a single point orbit under $G(R)$. Moreover, any two vertices of $R$ must have the same valency, $q$ (say). No vertex of $R$ can lie in the $xy$-plane or on the $z$-axis, respectively, since otherwise these invariant subspaces would have to contain all vertices of $R$. More generally, since $R$ is $3$-dimensional, its vertex-set cannot lie in a plane, and in particular cannot consist of a single point orbit under the cyclic subgroup $C_n$ of $G(R)$. This immediately rules out the possibility that $G(R)$ coincides with $C_n$ itself or with the (standard dihedral) group $H_{n}]C_{n}$. Thus these groups cannot be symmetry groups of any vertex-transitive polyhedron. 
\medskip

\noindent
{\it The groups $C_n^*$, $C_{2n}]C_{n}$ and $H_n$.}
\medskip

If $G(R)$ is among the three groups $C_n^*$, $C_{2n}]C_{n}$ and $H_n$ (of orders $2n$), the polyhedron $R$ must necessarily have exactly $2n$ vertices forming a single orbit under $G(R)$ and lying in two parallel planes $z=\pm t$, with $t>0$, each containing $n$ points. The two sets of $n$ points are the vertex-sets of two congruent convex regular $n$-gons, each with a center on the $z$-axis and invariant under the subgroup $C_n$. These $n$-gons are referred to as the {\em top $n$-gon\/} or {\em bottom $n$-gon\/}, respectively; they become the {\em top face\/} or {\em bottom face\/}, respectively, of the convex hull ${\rm conv}(R)$ of $R$. We also call a vertex of $R$ in the plane $z=t$ or $z=-t$, respectively, a {\em top vertex\/} or {\em bottom vertex\/}; each vertex of $R$ is of one of these kinds. We also use similar terminology for ${\rm conv}(R)$.

In most cases ${\rm conv}(R)$ is combinatorially an $n$-gonal antiprism; in some particular cases the top and bottom faces of ${\rm conv}(R)$ are polars of each other (up to translation), so then ${\rm conv}(R)$ is a standard antiprism. In all other cases ${\rm conv}(R)$ is an $n$-gonal right prism.

Now since the $2n$ vertices of $R$ lie on a sphere and are symmetrically positioned on a pair of parallel planes, it is immediately clear that, with possibly two exceptions, the faces of $R$ must be triangles or (convex) quadrangles; an exception occurs precisely when the top or bottom $n$-gon is a face of $R$. Moreover, an edge of $R$ joining two top vertices or two bottom vertices, must necessarily be an edge of the top or bottom $n$-gon, respectively; that is, an edge of $R$ cannot be a diagonal of the top or bottom $n$-gon. In fact, otherwise $R$ would have to have self-intersections, by the invariance of the top and bottom $n$-gons under the subgroup $C_n$ of $G(R)$. Thus an edge of $R$ must either be an edge of the top $n$-gon or the bottom $n$-gon, or pass from a top vertex to a bottom vertex. If an edge of the former kind occurs in $R$, then in fact all edges of the top $n$-gon and all edges of the bottom $n$-gon must occur as edges of~$R$, once again by the invariance of $R$ under $C_n$ and $G(R)$. 

Moreover, since a diagonal of the top or bottom $n$-gon cannot occur as an edge of $R$, every quadrangular face (if any) of $R$ must necessarily share one edge with the top $n$-gon and one edge with the bottom $n$-gon; these edges must necessarily be parallel, by the planarity of the quadrangular face. Now since no face of $R$ can pass through $o$ (otherwise $R$ would self-intersect at $o$), this then leaves only one possibility for quadrangular faces to occur, namely as the rectangular faces of the mantle of ${\rm conv}(R)$ when ${\rm conv}(R)$ is an $n$-gonal right prism. In this case exactly two (adjacent) rectangular faces meet at each vertex of $R$. Since $q\geq 3$, it also follows that $R$ must necessarily have triangular faces, unless $R$ itself is the boundary of an $n$-gonal right prism.

Now if $F$ is a triangular face of $R$, then $F$ has exactly two vertices in common with either the top $n$-gon or the bottom $n$-gon; also, the edge of $F$ connecting them must be an edge, not a diagonal, of this $n$-gon. It follows that each triangular face of $R$ must have an edge in common with the top $n$-gon or the bottom $n$-gon. Moreover, each edge of the top $n$-gon or bottom $n$-gon must occur as an edge of a triangular face of $R$, again by the invariance of the polyhedron under $C_n$ and $G(R)$. 

Now suppose $R$ has only triangular or quadrangular faces. Then we prove that $R$ must be a toroidal polyhedron, that is, $g=1$. Here we exploit the fact that the edge boundaries of the top and bottom $n$-gons serve as connectors for two parts of $R$, the ``inner shell" and the ``outer shell". Now suppose $u$ is a top vertex of $R$, and $v_1(u),\ldots,v_q(u)$, in this order, are  the vertices of $R$ adjacent to $u$. For $j=1,\ldots,q$, let $F_j(u)$ be the face of $R$ containing the vertices $v_j(u)$, $u$, $v_{j+1}(u)$ (and possibly a fourth vertex), with indices considered modulo $q$. Then exactly two vertices, $v_1(u)$ and $v_{1+k(u)}(u)$ (say), from among $v_1(u),\ldots,v_q(u)$ are also top vertices while all others are bottom vertices. We may choose our vertex labeling in such a way that each face from among $F_{1}(u),\ldots,F_{k(u)}(u)$ is a triangle, and that quadrangular faces of $R$ (if any) must occur among the faces $F_{1+k(u)}(u),\ldots,F_{q}(u)$. Now since $R$ is invariant under the cyclic subgroup $C_n$, we may further assume that the face labels associated with any two adjacent top vertices $u$ and $u'$ are consistent, in the sense that $k(u)=k(u')\geq 3$ and $F_{k(u)}(u) = F_{1}(u')$ (so $\{u,u'\}$ is an edge of either face). Thus $k:=k(u)$ is independent of $u$. Moreover, $k=3$; otherwise certain edges of $R$ would have to lie in more than two faces of $R$, namely the edges in the bottom $n$-gon if $k\geq 5$, or the edges of the form $F_{2}(u)\cap F_{3}(u)$ if $k=4$.

With these assumptions on the face labels in place, it then follows that the polyhedral subcomplex of $R$ made up of all faces $F_{1}(u),\ldots,F_{k(u)}(u)$, with $u$ a top vertex, is topologically a cylinder invariant under $C_n$ and bounded at the top and bottom, respectively, by the edge boundaries of the top and bottom $n$-gons; combinatorially, this subcomplex is the face subdivision of the mantle of an $n$-gonal antiprism. Similarly, the (closure of the) complement of this subcomplex in $R$ is another polyhedral subcomplex of $R$, again topologically a cylinder invariant under $C_n$; this second cyclinder shares the same boundary with the first cylinder, and its underlying subcomplex is isomorphic to the face subdivision of the mantle of an $n$-gonal antiprism or prism. These two subcomplexes form the two ``shells" of $R$, with the qualification ``inner" and ``outer" determined by our choice of labeling. Notice that each shell is also $G(R)$-invariant, not only $C_n$-invariant; in particular, the two shells cannot be interchanged by a symmetry in $G(R)$ outside of $C_n$. The invariance under $G(R)$ also shows that alternatively we can think of each shell as being obtained by a process parametrized by the bottom (rather than top) vertices $u$ of $R$, where as before certain faces from the vertex-stars are chosen and assembled to form topologically a cylinder. 

In summary, if $R$ has only triangular or quadrangular faces, then we have one of two possible scenarios. If $R$ has no quadrangular faces, then each shell is a suitably ``twisted" mantle of an $n$-gonal antiprism fitting together along the edge boundaries of the top and bottom $n$-gon to form a triangle-faced toroidal polyhedral; in particular, $R$ is equivelar of type $\{3,6\}$. If $R$ does have quadrangular faces, then the interior shell is again a suitably twisted mantle of an $n$-gonal antiprism, and the outer shell is the mantle of an $n$-gonal right prism. Again $R$ is a toroidal polyhedron, now with five faces at each vertex, namely three triangles and two rectangles. In either case $R$ is one of the toroidal vertex-transitive polyhedra described in \cite{gs}.

It remains to consider the case that $R$ has faces which are not triangles or quadrangles. Then there are exactly two such faces,  namely the top $n$-gon and the bottom $n$-gon; note here that the invariance under $G(R)$ forces both $n$-gons to occur. Leaving these two $n$-gonal faces aside for a moment (and bearing in mind that $n\geq 5$), we can proceed in much the same way as before and establish that the remaining faces of $R$ form a $G(R)$-invariant subcomplex which is topologically a cylinder bounded at the top and bottom by the edge boundaries of the two $n$-gonal faces of $R$. Depending on whether $R$ has quadrangular faces or not, this subcomplex is the mantle of an $n$-gonal right prism or a suitably twisted mantle of an $n$-gonal antiprism. In either case $R$ is a spherical polyhedron, with $g=0$. If $R$ has quadrangular faces, then all faces are rectangles or convex regular $n$-gons; all vertices are $3$-valent; and $R$ is the boundary of a $n$-gonal right prism. If $R$ has no quadrangular faces, then all faces are triangles or convex regular $n$-gons; all vertices are $5$-valent; and $R$ bounds a nonconvex solid, namely a suitably ``twisted" $n$-gonal antiprism.

We remark that not all the groups $C_n^*$, $C_{2n}]C_{n}$, and $H_n$ investigated here do actually occur as symmetry groups of vertex-transitive polyhedra of positive genus, that is, of genus~$1$. Apart from the fact than $n\geq 7$ is required for the construction of a twisted mantle of an $n$-gonal antiprism as described above, there are more severe restrictions arising from the presence of a horizontal plane of symmetry. For example, all groups $C_n^*$ with $n$ even and $C_{2n}]C_{n}$ with $n$ odd contain the reflection in the $xy$-plane as an element; but clearly, since self-intersections are not permitted, none of the above toroidal polyhedra admits a horizontal plane of symmetry. Moreover, none of the groups $H_n$ can occur as the symmetry group of a toroidal polyhedron. In fact, each half-turn in $H_n$ about a horizontal axis must necessarily intersect a twisted antiprismatic mantle for a polyhedron in an edge, forcing the two faces of the mantle meeting there to be coplanar; by the $C_n$ invariance, the polyhedron would have to have self-intersections. Thus the groups $C_n^*$ with $n$ odd and $C_{2n}]C_{n}$ with $n$ even are the only groups that remain, and these do actually occur; in fact, the parity of $n$ determines the group uniquely.
\medskip

\noindent
{\it The groups $H_n^*$ and $H_{2n}]H_{n}$.}
\medskip

We can reject the groups $H_n^*$ and $H_{2n}]H_{n}$ on the grounds that they contain plane reflections. In fact, if $s$ is a half-turn in $H_n$ or $H_{2n}$, then $r:=(-i)s$ is a reflection in the plane perpendicular to the rotation axis of $s$. For the groups $H_n^*$ this gives us $n$ reflections from the half-turns $s$ with rotation axis contained in the $xy$-plane, as well as the reflection in the $xy$-plane itself if $n$ is even and $s$ is the half-turn about the $z$-axis. Similarly, we obtain $n$ reflections from the half-turns $s$ in $H_{2n}\!\setminus\! H_n$ with rotation axis contained in the $xy$-plane, as well as the reflection in the $xy$-plane itself if $n$ is odd and $s$ is the half-turn (in $H_{2n}^*$) about the $z$-axis. In either case there are no other plane reflections in the group. Moreover, the $n$ reflection planes derived from the half-turns about lines in the $xy$-plane are perpendicular to the $xy$-plane and dissect $\E$ into $2n$ congruent angular regions, each bounded by two planes through the $z$-axis inclined by $\pi/n$; the intersections of the $n$ reflection planes with the $xy$-plane form the standard reflection line arrangement of a convex regular $n$-gon in the $xy$-plane. If the reflection in the $xy$-plane also belongs to the group under consideration, then this additional reflection plane further cuts the $n$ regions in space in half. 

Now suppose $R$ is a vertex-transitive polyhedron with symmetry group $G(R)$ given by $H_n^*$ and $H_{2n}]H_{n}$. Then $R$ has $2n$ or $4n$ vertices arranged in two sets of equal size in two parallel planes $z=\pm t$, with $t>0$; recall here that $C_n$ is a subgroup of $G(R)$ and has planar point orbits of size $n$. The number of vertices of $R$ is $2n$ if and only if at least one (and hence each) vertex of $R$ lies on one of the $n$ non-horizontal reflection planes. As before we use the same terminology of top and bottom vertices. Note also that the faces of $R$ can only be triangles and quadrangles, for the same reason as before.

We now exploit Lemma~\ref{edgecross}. First suppose that each vertex of $R$ lies on a non-horizontal reflection plane, so $R$ has $2n$ vertices. Then ${\rm conv}(R)$ is an $n$-gonal right prism. In this case some of the analysis for the previously discussed groups carries over. In particular, each vertex of $R$ has the same valency, $q$; no diagonal of the top $n$-gon can occur as an edge of $R$; and if $R$ has any quadrangular faces, then they must coincide with the mantle faces of the $n$-gonal prism. On the other hand, bearing in mind the crossing behavior of edges and reflection planes, we see that a top vertex of $R$ can only be joined to just one bottom vertex, namely the vertex directly below it. Thus $q=3$. Moreover, since all vertices of $R$ are $3$-valent, $R$ must actually bound a convex polytope and have genus $0$.

Now suppose no vertex of $R$ lies on a non-horizontal reflection plane. Then $R$ has $4n$ vertices, such that each of the $2n$ angular segments in the top plane and bottom plane contains exactly one vertex. Again, since crossing edges must be perpendicular to reflection plane, each top vertex of $R$ can only be joined to just one bottom vertex, namely the vertex in the angular segment of the bottom plane directly below it. Hence each top vertex must be joined to $q-1\geq 2$ other top vertices, where again $q$ denotes the valency of the vertices in $R$. We claim that self-intersections must occur if $q\geq 4$. First note that the $2n$ top vertices form a single orbit under the standard dihedral group of order $2n$ generated by the reflections in the $n$ non-horizontal mirrors; hence they are the vertices of a (generally non-regular) convex $2n$-gon in the top plane. Now, if $q-1\geq 3$, then any top vertex $u$ must be joined to another top vertex $v$ by an edge $e$ of $R$ which is not an edge of the top $2n$-gon. Suppose $u'$ is the top vertex which, as a vertex of the $2n$-gon, is adjacent to $u$ and lies on the (combinatorially) shorter boundary arc of the $2n$-gon connecting $u$ and $v$ (if $v$ is diametrically opposite to $u$ in the $n$-gon, we choose either of the two possible arcs). Then the perpendicular bisector of $u$ and $u'$ is a reflection plane for the group and the corresponding reflection takes $e$ to another edge intersecting $e$ non-trivially. Thus again $q=3$, and $R$ must bound  a convex polyhedron and have genus $0$.
\medskip

\noindent
{\bf Conclusions}
\medskip

In summary, we established the following two theorems.

\begin{thm}
\label{verpolfin}
There are only finitely many vertex-transitive polyhedra in the genus range $g\geq 2$. The symmetry group of each vertex-transitive polyhedron in the genus range $g\geq 2$ is a Platonic rotation group.
\end{thm}

\begin{thm}
\label{genus1}
A vertex-transitive polyhedron with a reducible symmetry group must have genus $0$ or $1$. There are infinitely many vertex-transitive polyhedra of genus $0$, as well as of genus $1$. 
\end{thm}

\end{document}